\documentclass[]{interact}
 \usepackage{color}
 \usepackage{bm}
\usepackage{epstopdf}
\usepackage{caption}
\usepackage{graphicx}
\usepackage{subcaption}
\usepackage{algorithm}
\usepackage{algorithmicx}%
\usepackage{algpseudocode}
\usepackage{listings}

\usepackage{amssymb}
\usepackage{amsfonts}
\usepackage{amsmath}
\usepackage{enumerate,amstext,graphicx,latexsym,amscd}
\usepackage[titletoc, title]{appendix}
\usepackage{bbding}
\usepackage[misc]{ifsym}
\usepackage{ifthen,calc,lastpage,listings}
\usepackage{fancyhdr}
\usepackage{footmisc}
\usepackage[colorlinks,linkcolor=blue,citecolor=blue]{hyperref}
\usepackage[capitalise, nosort]{cleveref}
\crefname{equation}{}{}
\crefname{lem}{Lemma}{Lemmas}
\crefname{thm}{Theorem}{Theorems}
\crefname{assum}{Assumption}{Assumptions}
  \usepackage{paralist}
  \usepackage{graphics} 
  \usepackage{epsfig} 
 \usepackage{epstopdf}
\usepackage{geometry}
\usepackage{dsfont}
\usepackage{float}
 \usepackage{booktabs,multirow}
\usepackage[numbers,sort&compress]{natbib}
\bibpunct[, ]{[}{]}{,}{n}{,}{,}

\theoremstyle{plain}
\newtheorem{theorem}{Theorem}[section]
\newtheorem{lemma}[theorem]{Lemma}
\newtheorem{corollary}[theorem]{Corollary}

\theoremstyle{definition}
\newtheorem{definition}[theorem]{Definition}

\theoremstyle{remark}
\newtheorem{remark}{Remark}

\newtheorem{assum}{Assumption}
\newcommand{\proj}[0]{ {\bf proj}}
\newcommand{\conv}[1]{{\bf conv}\left\{ {#1} \right\}}
\newcommand{\R}{\,{\mathbb R}}

\begin{document}


\title{Point Convergence Analysis of the Accelerated Gradient Method for Multiobjective Optimization: Continuous and Discrete}

\author{
\name{Yingdong Yin\textsuperscript{a}}
\affil{\textsuperscript{a}National Center for Applied Mathematics in Chongqing, Chongqing Normal University, Chongqing, China}
}

\maketitle

\begin{abstract}
This paper investigates the point convergence of accelerated gradient methods for multiobjective optimization, in both continuous and discrete settings. We address the open problems of whether the solution trajectory of the multiobjective inertial gradient-like dynamical system (MAVD) with asymptotic vanishing damping converges when \(\alpha = 3\), and whether the sequence generated by the multiobjective Nesterov accelerated method (MAG) converges to a weakly Pareto optimal solution. For the continuous system (MAVD) with \(\alpha = 3\), we prove that the trajectory \(x(t)\) converges to a weakly Pareto optimal solution. For the discrete case, we propose a multiobjective accelerated gradient method with a generalized momentum factor (MAG-GM), and prove that the generated sequence \(\{x_k\}\) converges to a weakly Pareto optimal solution.
\end{abstract}

\section{Introduction}
In the setting of finite-dimensional real vector space $\mathbb{R}^n$ and the general Euclidean inner product $\langle\cdot ,\cdot \rangle$, this paper considers the following unconstrained multiobjective optimization problem:
\begin{equation}\label{eq:MOP}
    \min_{x\in \mathbb{R}^n}F(x):=(f_1(x),\cdots, f_m(x)),\tag{MOP}
\end{equation}
where each objective function $f_i:\mathbb{R}^n\to \mathbb{R}$ is continuously differentiable. Solving \cref{eq:MOP} involves seeking a solution in the sense of Pareto optimality. When $m=1$, this solution is also the optimal solution in the usual sense for single-objective optimization. In the general convex case, many algorithms for solving \cref{eq:MOP} often only obtain a weakly Pareto optimal solution; roughly speaking, no feasible point in $\mathbb{R}^n$ can be found that is strictly worse than this solution.

In recent years, iterative algorithms for solving \cref{eq:MOP} have been extensively studied. Filege et al. proposed the multiobjective steepest descent method (MSD) \cite{fliege2000steepest}, which obtains the multiobjective steepest descent direction  by solving a quadratic programming problem. For multiobjective composite optimization problems, Tanabe et al. proposed a class of multiobjective proximal gradient methods (MPG) for solution \cite{tanabe2019proximal}. Under evaluation by a merit function proposed by the authors themselves \cite{tanabe2024new}, this method achieves the same sublinear convergence rate $O(1/k)$ as MSD \cite{tanabe2023convergence}. Based on this, Tanabe et al. further proposed the multiobjective accelerated proximal gradient method (MAPG) \cite{tanabe2023accelerated}, ultimately achieving a faster sublinear convergence rate $O(1/k^2)$. This convergence rate is consistent with the Nesterov accelerated method \cite{nesterov2018lectures} and the FISTA \cite{beck2017first} in single-objective optimization. Both the steepest descent method and the accelerated method are referred to as first-order algorithms because the iterative point sequence does not use curvature information. For non-first-order multiobjective iterative algorithms, refer to \cite{fliege2009newton,chen2023barzilai,morovati2016barzilai,prudente2024global,povalej2014quasi,wang2019extended,ansary2015modified,qu2011quasi}.

Using gradient dynamical systems provides a new perspective for studying first-order algorithms, and their convergence analysis often shares similar analytical processes with discrete algorithms. Attouch et al. proposed the following \textit{multiobjective gradient system} as the continuous version of MSD \cite{Attouch2014}:
\begin{equation}\label{eq:MOG}
    \dot x(t)+\proj_{C(x(t))}(0)=0,\tag{MOG}
\end{equation}
where $C(x):=\conv{\nabla f_i(x):i=1,\cdots,m}$ and $\proj_{C(x)}(0)$ denotes the projection of the zero vector onto $C(x)$. Using a merit function, the sublinear convergence rate $O(1/t)$ of the solution trajectory of \cref{eq:MOG} can be characterized \cite{yin2025multiobjective2}, which corresponds to the convergence rate result of MPG. Attouch et al. also proposed a class of accelerated gradient dynamical systems \cite{attouch2015multiibjective}. Based on this, Sonntag et al. \cite{sonntag2024fastSIAMOptimization} proposed the following second-order equation, the \textit{multiobjective
 inertial gradient-like dynamical system with asymptotic vanishing damping}:
\begin{equation}\label{eq:MAVD}
    \frac{\alpha }{t}\dot x(t)+\proj_{C(x(t))+\ddot x(t)}(0)=0,\tag{MAVD}
\end{equation}
where $\alpha >0$. When $\alpha \ge 3$, the solution trajectory of \cref{eq:MAVD} achieves a faster sublinear convergence rate $O(1/t^2)$, and when $\alpha > 3$, it converges to a weakly Pareto optimal solution. For other second-order multiobjective gradient dynamical systems, refer to \cite{boct2024inertial,luo2025accelerated,yin2025multiobjective}. When $\alpha = 3$, Sonntag provided a discrete version of \cref{eq:MAVD}, which is also the multiobjective Nesterov accelerated method, with the following iterative scheme:
\begin{equation}\label{eq:MAG}
    \left\{\begin{aligned}
        y_k&=x_k+\frac{k-1}{k+2}(x_k-x_{k-1}),\\
        x_{k+1}& = y_k -s\proj_{C(y_k)}\left(\frac{1}{s}(y_k-x_k)\right).
    \end{aligned}\right.\tag{MAG}
\end{equation}
This algorithm achieves a faster sublinear convergence rate of $O(1/k^2)$, and any cluster point of the iterative sequence is a weakly Pareto optimal solution. Therefore, two natural and corresponding problems remain unresolved:
\begin{itemize} 
\item \textit{When $\alpha = 3$, does the solution trajectory of \cref{eq:MAVD} converge to a weakly Pareto optimal solution?} 
\item \textit{Does the generated sequence of \cref{eq:MAG} converge to a weakly Pareto optimal solution?}
\end{itemize} 

In fact, when $m=1$, \cref{eq:MAVD} reduces to the following dynamical system:
\begin{equation}\label{eq:AVD}
    \ddot x(t)+\frac{\alpha }{t}+\nabla f(x(t))=0,\tag{AVD}
\end{equation}
which was first proposed by Su et al. \cite{su2016differential}, and it and its variants have been extensively studied \cite{wang2021search,attouch2022damped,attouch2022first,attouch2018fast,attouch2019fast,chen2019first,luo2022differential}. \cref{eq:AVD} corresponds to the Nesterov accelerated algorithm, and both also share a long-standing unresolved problem: whether the solution trajectory and the generated sequence converge when $\alpha = 3$, respectively. This problem was recently solved independently by Jang and Ryu \cite{Jang2025Point} and Bo\c{t} et al. \cite{boct2025iteratesnesterovsacceleratedalgorithm}. The key to solving this problem lies in analyzing the original Lyapunov function from a different perspective, thereby establishing a connection between the cluster points of the solution trajectory (generated sequence) and the limit of the Lyapunov function. Building on their methods and overcoming the difficulties caused by the non-negativity and non-monotonic decrease of the Lyapunov function in multiobjective optimization, we successfully resolve the aforementioned problems. Specific contributions are as follows:

\begin{itemize}
    \item \textbf{Trajectory convergence}: For the case $\alpha = 3$ in \cref{eq:MAVD}, we provide a convergence analysis of its trajectory $x(t)$, proving that $\lim_{t\to \infty }x(t)=x^*\in \mathbb{R}$ exists and that $x^*$ is a weakly Pareto optimal solution of \cref{eq:MOP}.
    \item \textbf{Point convergence}: For the algorithm \cref{eq:MAG}, we adopt a more general parameter selection method to construct a new algorithm. This selection method refers to \cite{tanabe2022globally}, with the specific iterative scheme as follows:
    \begin{equation}\label{eq:intro-algorithm-iterschme}
    \left\{\begin{aligned}
    t_{k+1} &=\sqrt{t_k^2 - a t_k + b} + \frac{1}{2}, \\
    y_k &= x_k + \frac{t_k - 1}{t_{k+1}}(x_k - x_{k-1}), \\
    x_{k+1} &= y_k - s \cdot \proj_{C(y_k)}\left(\frac{1}{s}(y_k-x_k)\right).
    \end{aligned}\right.  
    \end{equation}
    We prove that the iterative sequence $\{x_k\}$ generated by \cref{eq:intro-algorithm-iterschme} satisfies $\lim_{k\to \infty }x_k=x^*\in \mathbb{R}$ and that $x^*$ is a weakly Pareto optimal solution of \cref{eq:MOP}.
\end{itemize}
The paper is structured as follows. \cref{sec:Preliminary} introduces the necessary notations, concepts of Pareto optimality, and the foundational assumptions. The convergence analysis of \cref{eq:MAVD} for the case $\alpha =3$ is presented in \cref{sec:Dyanmical-System}. \cref{sec:algo} is dedicated to the discrete-time setting, where we propose a multiobjective accelerated gradient algorithm with generalized momentum and analyze its convergence properties. 
\section{Preliminary}\label{sec:Preliminary}
\subsection{Notation}
In this paper, $\mathbb{R}^d$ denotes a $d$-dimensional Euclidean space with the inner product $\langle \cdot, \cdot \rangle$ and the induced norm $\|\cdot\|$. For any vectors $a, b \in \mathbb{R}^d$, we say $a \leq b$ if $a_i \leq b_i$ holds for all $i = 1, \ldots, d$; the relations $<$, $\geq$, and $>$ are defined analogously. The nonnegative  and positive orthants are denoted respectively by
$$\mathbb{R}_{+}^d:=\{x\in \mathbb R^d : x\geq 0\}, ~~\mathbb{R}^d_{++}:=\{x\in \mathbb{R}^d:x>0\}.$$ 
The set $\Delta^m := \{ \theta \in \mathbb{R}^m : \theta \geq 0 \text{ and } \sum_{i=1}^m \theta_i = 1 \}$ is the positive unit simplex. Given a set of vectors $\{\eta_1, \ldots, \eta_m\} \subseteq \mathbb{R}^d$, their convex hull is defined as $\conv {\eta_1, \ldots, \eta_m} := \{ \sum_{i=1}^m \theta_i \eta_i : \theta \in \Delta^m \}$. For a closed convex set $C \subseteq \mathbb{R}^d$, the projection of a vector $x$ onto $C$ is  $\proj_{ C}(x) := \arg\min_{y \in C} \|y - x\|^2$. 
\subsection{Pareto optimality}
In multiobjective optimization, a single solution that minimizes all conflicting objectives simultaneously is generally unavailable. Instead, we seek Pareto optimal solutions, which represent the best possible trade-offs. The following definitions formalize this concept of optimality.
\begin{definition}[\cite{miettinen1999nonlinear}]\label{def:defofpareto}
    Consider the multiobjective optimization problem \cref{eq:MOP}.
    \begin{itemize}
        \item[$\rm (i)$] A point $x^* \in \mathbb{R}^n$ is a Pareto optimal solution  if there exists no \(y\in \mathbb R^n\) that \(F(y)\leq F(x^*)\) and $F(y)\neq F(x^*)$. The set of all Pareto optimal solutions is called the Pareto set, denoted by $\mathcal{P}$.  Its image under $F$, denoted by \(F(\mathcal P)\), is called  the Pareto front.
        \item[$\rm (ii)$] A point $x^*\in \mathbb{R}^n$ is a locally Pareto optimal solution if there exists $\delta >0$ such that $x^*$ is Pareto optimal in $B_{\delta}(x^*)$. 
        \item[$\rm (iii)$] A point $x^* \in \mathbb{R}^n$ is a weakly Pareto optimal solution if there exists no \(y\in \mathbb R^n\) that \(F(y)<F(x^*)\). The set of all weakly Pareto optimal solutions is called the weak Pareto set and is denoted by $\mathcal{P}_w$. Its image under $F$, denoted by \(F(\mathcal P_w)\),  is called the weak Pareto front.
        \item[$\rm (iv)$]  A point $x^*\in \mathbb{R}^n$ is a locally weaklly Pareto optimal solution if there exists $\delta >0$ such that $x^*$ is weakly Pareto optimal in $B_{\delta}(x^*)$. 
    \end{itemize}  
\end{definition}

\begin{definition}\label{def:KKTcondition}
    A point $x^* \in \mathbb{R}^n$ is called Pareto critical if there exists $\theta \in \Delta^m $ satisfying
    \begin{equation}\label{eq:KKTpoint}
    \sum_{i=1}^m \theta_i \nabla f_i(x^*) = 0.
    \end{equation}
    The set of all Pareto critical points is called the Pareto critical set and is denoted by $\mathcal{P}_c$.
\end{definition}

 \begin{lemma}
    The following statements hold:
    \begin{itemize}
        \item[$\rm(i)$] If $x^*\in \mathbb{R}^n$ is locally weakly Pareto optimal for \cref{eq:MOP}, then $x^*$ is Pareto critical for \cref{eq:MOP};
        \item[$\rm(ii)$] If each component function $f_i$ $(i = 1, 2, \cdots, m)$ is convex and $x^*\in \mathbb{R}^n$ is Pareto  critical for \cref{eq:MOP}, then $x^*$ is weakly Pareto optimal for \cref{eq:MOP};
        \item[$\rm(iii)$] If each component function $f_i$ $(i = 1, 2, \cdots, m)$  is strictly convex and $x^*\in \mathbb{R}^n$ is Pareto critical for \cref{eq:MOP}, then $x^*$ is Pareto optimal for \cref{eq:MOP}. 
    \end{itemize}
\end{lemma}
\subsection{Merit function}

In single-objective optimization, the function value $f(x) - f(x^*)$ is commonly used to measure the progress of an algorithm. For \cref{eq:MOP}, however, there is no single optimal value $f(x^*)$. To address this, a merit function $u_0(x)$ is introduced to play an analogous role \cite{tanabe2024new}:
\begin{equation}\label{eq:meritfunction}  
u_0(x) := \sup_{z \in \mathbb{R}^n} \min_{i=1,\ldots,m} \left( f_i(x) - f_i(z) \right).  
\end{equation}  
It serves as a valuable tool for quantifying convergence behavior in the multiobjective setting, as detailed in the following  theorem.
\begin{theorem}[\cite{tanabe2024new}]\label{thm:weakpareto} 
    Let $u_0:\mathbb R^n \to \mathbb{R}$ be defined as in \cref{eq:meritfunction}. Then,
    \begin{itemize} 
        \item[\(\rm (i)\)]$u_0(x) \geq 0$ for all $x \in \mathbb{R}^n$;
        \item[\(\rm (ii)\)] $x \in \mathbb{R}^n$ is a weakly Pareto optimal solution of \cref{eq:MOP} if and only if $u_0(x) = 0$;
        \item[\(\rm (iii)\)] $u_0(x)$ is lower semicontinuous.
    \end{itemize}
\end{theorem}
\subsection{Assumption}
The convergence analysis in this work relies on the following standard assumptions regarding the objective functions and the solution set.
\begin{assum} \label{assum:Lj-muj} Each function $f_i:\mathbb{R}^n\to \mathbb{R}$, $i=1,\cdots,m$ is convex and continuously differentiable, with Lipschitz gradients, i.e., {there exist $L_i\in \R$ such that } $\|\nabla f_i({x})-\nabla f_i({y})\|\le L_i\|{x}-{y}\|$
    for all  ${x},{y}\in \mathbb{R}^n$. Let $L:\max_{i=1,\cdots,m} L_i$. 
\end{assum}
     
\begin{assum}\label{assum:levelset}	The level set $\mathcal{L}(F,F(x^*))=\{{x}\in\mathbb{R}^n :F({x})\le F(x^*)\}$ is bounded for any $x^*\in \mathbb{R}^n$. Let
    $$R_{x^*}:=\sup_{x\in \mathcal{L}(F, F(x^*))}\|x\|<+\infty.  $$
\end{assum}
\begin{remark}
\cref{assum:Lj-muj} is a general assumption regarding the objective function. \cref{assum:levelset} is stronger than those imposed on the solution set in \cite{sonntag2024fastSIAMOptimization,sonntag2024fast,tanabe2023accelerated,tanabe2019proximal}. In fact, the subsequent conclusions on convergence rates in this paper can hold under weaker conditions than \cref{assum:levelset}. However, the proofs of the core results in this paper—the convergence of trajectories and point sequences—rely on \cref{assum:levelset}.
\end{remark}
\section{Dynamical system}\label{sec:Dyanmical-System}
Consider the following Cauchy problem  
\begin{equation}\label{eq:CP} 
\left\{
\begin{aligned}
&\frac{\alpha }{t}\dot x(t) + \proj_{C(x(t)) + \ddot x(t)}(0) = 0, \\
&x(1) = x_0, ~~ \dot x(1) = 0,
\end{aligned}
\right.\tag{CP}
\end{equation}
where $\alpha > 0$. The proof of the existence of a solution to \cref{eq:CP} has been provided in \cite{sonntag2024fastSIAMOptimization}, meaning that there exists a smooth function $x:[1, +\infty) \to \mathbb{R}^n$ that satisfies \cref{eq:CP} for almost all $t \ge 1$, and its derivative $\dot x$ is absolutely continuous on any closed interval of $[1, +\infty)$. Moreover, under the condition that the objective functions $f_i$ are all bounded from below, the smooth function $x(t)$ satisfies $\lim_{t \to \infty} u_0(x(t)) = 0$ and the existence of $\lim_{t \to \infty} f_i(x(t)) = f_i^\infty \in \mathbb{R}$. For the case $0 < \alpha \le 3$, we adopt different proof techniques, which share a similar analytical process to the convergence analysis of discrete algorithms in \cref{sec:algo}.

Let $ x: [1, +\infty) \to \mathbb{R}^m $ be a solution to the problem \cref{eq:CP}. For $ i = 1, \cdots, m $, define the following energy function:  
\begin{equation}\label{eq:energy-function-W}
\mathcal{W}_i :=[1,+\infty )\to \mathbb{R},~~t\mapsto  f_i(x(t)) + \frac{1}{2} \|\dot{x}(t)\|^2.
\end{equation}

\begin{lemma}\label{lem:Dynamical-system-lemma1}
Suppose \cref{assum:Lj-muj} holds. Let $ x: [1, +\infty) \to \mathbb{R}^n $ be a solution to the problem \cref{eq:CP}, and let $ \mathcal{W}_i $ be defined as in \cref{eq:energy-function-W}. Then, 
\begin{itemize}    
\item[(i)] $ \mathcal{W}_i(t) $ is monotonically non-increasing for $ i = 1, \cdots, m $;  
\item[(ii)] Assuming $ f_i $ is bounded, then $ \lim_{t \to \infty} \mathcal{W}_i(t) = \mathcal{W}_i^* \in \mathbb{R} $ exists, and  
$$
\int_{1}^{\infty} \frac{1}{t} \|\dot{x}(t)\|^2 \, dt < +\infty.
$$
\end{itemize}
\end{lemma} 
\begin{proof}
(i) According to the projection theorem, we have  
$$
\left\langle \nabla f_i(x(t)) + \ddot{x}(t) + \frac{\alpha}{t} \dot{x}(t), \dot{x}(t) \right\rangle \le 0.
$$
Therefore, we have for any $ i = 1, \cdots, m $,  
\begin{equation} \label{eq:energy-function-pro1}
\begin{aligned}
\frac{d}{dt} \mathcal{W}_i(t) + \frac{\alpha}{t} \|\dot{x}(t)\|^2 &= \langle \nabla f_i(x(t)), \dot{x}(t) \rangle + \langle \ddot{x}(t), \dot{x}(t) \rangle + \left\langle \frac{\alpha}{t} \dot{x}(t), \dot{x}(t) \right\rangle \\
&= \left\langle \nabla f_i(x(t)) + \ddot{x}(t) + \frac{\alpha}{t} \dot{x}(t), \dot{x}(t) \right\rangle \le 0.
\end{aligned}
\end{equation}
Hence, $ \mathcal{W}_i(t) $ is monotonically non-increasing.  

(ii) Since $ f_i $ is bounded, $ \mathcal{W}_i(t) $ is bounded. Consequently, $ \mathcal{W}_i(t) \to \inf_{t \ge 1} \mathcal{W}_i(t) := \mathcal{W}_i^* \in \mathbb{R} $ as $ t \to +\infty $. Integrating \cref{eq:energy-function-pro1} over $ [1, r] $ for any $ r > t $, and rearranging terms, we obtain  
$$
\int_{1}^{r} \frac{\alpha}{t} \|\dot{x}(t)\|^2 \, dt \le \mathcal{W}_i(1) - \mathcal{W}_i(r) \le \mathcal{W}_i(1) - \mathcal{W}_i^*.
$$
Letting $ r \to +\infty $, the conclusion follows.
\end{proof}
\begin{corollary}\label{coro:Dynamical-system-corollary1}
Suppose \cref{assum:Lj-muj} holds. Let $ x: [1, +\infty) \to \mathbb{R}^n $ be a solution to the problem \cref{eq:CP}. Then, for $i=1,\cdots,m$,
$$f_i(x(t))\le f_i(x_0),$$ for $t\ge 1$. 
\end{corollary}
\begin{proof}
By \cref{lem:Dynamical-system-lemma1}, we have
$$f_i(x(t))\le \mathcal W_i(t)\le \mathcal W_i(1)=f_i(x_0),$$
for $t\ge 1$. 
\end{proof}
\begin{remark}
\cref{coro:Dynamical-system-corollary1} shows that $x(t) \subseteq \mathcal{L}(F, F(x_0))$ for $t \ge 1$, so when \cref{assum:levelset} holds, $x: [1, +\infty) \to \mathbb{R}^n$ is a bounded function.
\end{remark}
\subsection{Convergence rate}
Let $x: [1, +\infty) \to \mathbb{R}^n$ be a solution of the CP, and define
\begin{equation}\label{eq:Theta_z} 
\Theta_z:[1,+\infty )\to \mathbb{R},~~ t\mapsto  \min_{i=1,\cdots,m} \left( f_i(x(t)) - f_i(z) \right), 
\end{equation}
and the auxiliary function  
\begin{equation}\label{eq:aux-function} 
\begin{aligned}
&\mathcal E_z:[1,+\infty )\to \mathbb{R},\\&t\mapsto  t^{\frac{2\alpha}{3}} \Theta_z(t) + \frac{1}{2} t^{\frac{2\alpha}{3} - 2} \left\| t \dot{x}(t) + \frac{2\alpha}{3} (x(t) - z) \right\|^2 + \frac{r(3 - \alpha)}{9} t^{\frac{2\alpha}{3} - 2} \|x(t) - z\|^2,
\end{aligned}
\end{equation} 
for any $z\in \mathbb{R}^n$. Since $\dot{x}(t)$ is absolutely continuous, a straightforward calculation shows that for almost all $t \in [1, +\infty)$,  
\begin{equation}\label{eq:aux-function-p1}
\begin{aligned}
&~~~~~ t^{3 - \frac{2\alpha}{3}} \frac{d}{dt} \left[ \frac{1}{2} t^{\frac{2\alpha}{3} - 2} \left\| t \dot{x}(t) + \frac{2\alpha}{3} (x(t) - z) \right\|^2 \right] \\
&= \left( \frac{\alpha}{3} - 1 \right) \cdot \left( \frac{2\alpha}{3} \|x(t) - z\|^2 + \left\langle t \dot{x}(t), \frac{2\alpha}{3} (x(t) - z) \right\rangle \right) \\
&\quad + t^2 \left\langle t \dot{x}(t) + \frac{2\alpha}{3} (x(t) - z), \frac{\alpha}{t} \dot{x}(t) + \ddot{x}(t) \right\rangle,
\end{aligned}
\end{equation}
and 
\begin{equation}\label{eq:aux-function-p2}
\begin{aligned}
& ~~~~~t^{3 - \frac{2\alpha}{3}} \frac{d}{dt} \left[ \frac{\alpha(3 - \alpha)}{9} t^{\frac{2\alpha}{3} - 2} \|x(t) - z\|^2 \right] \\
&= \left( 1 - \frac{\alpha}{3} \right) \cdot \left( \frac{2\alpha}{3} \|x(t) - z\|^2 + \left\langle t \dot{x}(t), \frac{2\alpha}{3} (x(t) - z) \right\rangle \right) \\
&\quad + \left( 1 - \frac{\alpha}{3} \right) \cdot \left( \frac{\alpha }{3} - 2 \right) \cdot \frac{2\alpha }{3} \|x(t) - z\|^2.
\end{aligned}
\end{equation}
\begin{lemma}\label{lem:Dynamical-system-lemma2}
Suppose \cref{assum:Lj-muj} holds. Let $ x: [1, +\infty) \to \mathbb{R}^n $ be a solution to the problem \cref{eq:CP} with $\alpha \in(0,3]$, and let $ \mathcal{E}_z(t) $ be defined as in \cref{eq:aux-function}. Then, $\mathcal{E}_z(t)$ is  differentiable for almost all  $t\in [1, +\infty)$, and it is monotonically non-increasing.
\end{lemma}
\begin{proof}
For any closed interval $[T_1, T_2] \subseteq [1, +\infty)$, and for any two points $t_1, t_2 \in [T_1, T_2]$, we have  
$$
\begin{aligned}
|\Theta_z(t_1) - \Theta_z(t_2)| &\le \max_{i=1,\cdots,m} |f_i(x(t_1)) - f_i(x(t_2))| \\
&\le \max_{i=1,\cdots,m} \max_{\xi \in x([T_1, T_2])} \|\nabla f_i(\xi)\| \cdot |x(t_1) - x(t_2)|,
\end{aligned}
$$  
so $\Theta_z(t)$ is absolutely continuous on $[T_1, T_2]$, and thus $\mathcal{E}_z(t)$ is almost everywhere differentiable on $[1, +\infty)$. According to Lemma 4.12 in \cite{sonntag2024fastSIAMOptimization}, for almost all $t \in [1, +\infty)$, there exists $i_0 \in \{1, \cdots, m\}$ such that  
$$
\frac{d}{dt} \Theta_z(t) = \frac{d}{dt} \left( f_{i_0}(x(t)) - f_{i_0}(z) \right) = \langle \nabla f_{i_0}(x(t)), \dot{x}(t) \rangle,
$$  
and $\Theta_z(t) = f_{i_0}(x(t)) - f_{i_0}(z)$. Therefore,  
\begin{equation}
\label{eq:aux-function-p3}
t^{3 - \frac{2\alpha}{3}} \frac{d}{dt} \left[ t^{\frac{2\alpha}{3}} \Theta_z(t) \right] = \frac{2\alpha}{3} t^2 \Theta_z(t) + t^2 \langle t \dot{x}(t), \nabla f_{i_0}(x(t)) \rangle.
\end{equation}  
Combining \cref{eq:aux-function-p1,eq:aux-function-p2,eq:aux-function-p3}, for almost all $t \in [1, +\infty)$, we have  
\begin{equation}\label{eq:aux-function-inequal1}
\begin{aligned}
&~~~~~ t^{3 - \frac{2\alpha}{3}} \frac{d}{dt} \mathcal{E}_z(t) \\
&= \frac{2\alpha}{3} t^2 \Theta_z(t) + t^2 \langle t \dot{x}(t), \nabla f_{i_0}(x(t)) \rangle + t^2 \left\langle t \dot{x}(t) + \frac{2\alpha}{3} (x(t) - z), \frac{\alpha}{t} \dot{x}(t) + \ddot{x}(t) \right\rangle \\
&\quad + \left(1 - \frac{\alpha}{3} \right) \cdot \left( \frac{\alpha}{3} - 2 \right) \cdot \frac{2\alpha}{3} \|x(t) - z\|^2 \\
&= \frac{2\alpha}{3} t^2 \Theta_z(t) + t^2 \left\langle t \dot{x}(t), \nabla f_{i_0}(x(t)) + \frac{\alpha}{t} \dot{x}(t) + \ddot{x}(t) \right\rangle \\&\quad + \frac{2\alpha}{3} t^2 \left\langle x(t) - z, -\proj_{C(x(t))}(-\ddot{x}(t)) \right\rangle + \left(1 - \frac{\alpha}{3} \right) \left( \frac{\alpha}{3} - 2 \right) \frac{2\alpha}{3} \|x(t) - z\|^2,
\end{aligned}
\end{equation} 
where the second equality follows from the form of the equation in \cref{eq:CP}. Furthermore, since there exists $(\lambda_1(t), \cdots, \lambda_m(t)) \in \Delta^m$ such that  
\begin{equation*}
\proj_{C(x(t))}(-\ddot{x}(t)) = \sum_{i=1}^m \lambda_i(t) \nabla f_i(x(t)),
\end{equation*}
we have  
\begin{equation}\label{eq:aux-function-inequal2}
\begin{aligned} 
\left\langle x(t) - z, -\proj_{C(x(t))}(-\ddot{x}(t)) \right\rangle &= \left\langle x(t) - z, -\sum_{i=1}^m \lambda_i \nabla f_i(x(t)) \right\rangle \\
&\le -\min_{i=1,\cdots,m} (f_i(x(t)) - f_i(z)) = -\Theta_z(t).
\end{aligned}
\end{equation} 
Combining \cref{eq:aux-function-inequal1,eq:aux-function-inequal2}, we obtain for all $t\ge 1$,  
$$
\frac{d}{dt} \mathcal{E}_z(t) \le 0.
$$  
Thus, for any $t_1 \le t_2$,  
$$
\mathcal{E}_z(t_2) - \mathcal{E}_z(t_1) = \int_{t_1}^{t_2} \frac{d}{dt} \mathcal{E}_z(t) \, dt \le 0.
$$  
Therefore, $\mathcal{E}_z(t)$ is monotonically decreasing on $[1, +\infty)$.
\end{proof}
 The following theorem establishes the sublinear convergence rate of the merit function $u_0(x(t))$ along the trajectory, which is a key result for the continuous-time case.
\begin{theorem}\label{thm:Dynamical-system-convergence-rate}
Suppose \cref{assum:Lj-muj,assum:levelset} hold. Let $ x: [1, +\infty) \to \mathbb{R}^n $ be a solution to the problem \cref{eq:CP} with $\alpha \in(0,3]$. Then, $u_0(x(t))=O(1/t^{\frac{2\alpha }{3}})$.
\end{theorem}
\begin{proof}
According to \cref{lem:Dynamical-system-lemma2}, we have  
$$
t^{\frac{2\alpha}{3}} \Theta_z(t) \le \mathcal{E}_z(t) \le \mathcal{E}_z(1) \le u_0(x_0) + \frac{\alpha(\alpha+3)}{9} \|x(t) - z\|^2.
$$
Taking the supremum over $z$ in $\mathcal{L}(F, F(x_0))$ on both sides, we obtain  
$$
t^{\frac{2\alpha}{3}} u_0(x(t)) \le u_0(x_0) + \frac{2\alpha(\alpha + 3)}{9} R_{x_0}^2.
$$
We complete the proof. 
\end{proof}
\subsection{Point convergence with $\alpha =3$}
\begin{lemma}\label{lem:Dynamical-system-point-convergence-1}
    Suppose \cref{assum:Lj-muj,assum:levelset} hold. Let $ x: [1, +\infty) \to \mathbb{R}^n $ be a solution to the problem \cref{eq:CP} with $\alpha \in(0,3]$. Suppose $f_i$ is bounded below for $i=1,\cdots,m$, then, 
    \begin{itemize}
        \item[(i)] $\lim_{t\to \infty }f_i(x(t))=f_i^\infty$ exists;
        \item[(ii)]  Let $z^*$ be any cluster point of $x$. Then for any $k \ge 1$,  
        $$-\Theta_{z^*}(t) \le \frac{1}{2} \|\dot x(t)\|^2.$$
    \end{itemize}
\end{lemma}
\begin{proof}
(i) Since $\mathcal W_i(t)$ is monotonically non-increasing, for any $t \ge s\ge 1$ and any $z \in \mathbb{R}^n$, we have  
$$
\mathcal W_i(t) - f_i(z) \le \mathcal W_i(s) - f_i(z).
$$  
From the definition of $\mathcal W_i(t)$, it follows that  
$$
\Theta_z(t) + \frac{1}{2}\|\dot x(t)\|^2 \le \Theta_z(s) + \frac{1}{2}\|\dot x(s)\|^2.
$$  
Taking the supremum over $z \in \mathbb{R}^n$ on both sides, we obtain  
$$
\mathcal W(t) := u_0(x(t)) + \frac{1}{2}\|\dot x(t)\|^2 \le u_0(x(s)) + \frac{1}{2}\|\dot x(s)\|^2.
$$  
Thus, $\lim_{t \to \infty} \mathcal W(t) = \mathcal W^*$ exists. Since $\lim_{t \to \infty} u_0(x(t)) = 0$ by \cref{thm:Dynamical-system-convergence-rate}, it follows that $\lim_{t \to \infty} \|\dot x(t)\|^2$ exists. According to \cref{lem:Dynamical-system-lemma1}, we further obtain  
$$
\lim_{t \to \infty} \|\dot x(t)\|^2 = 0.
$$  
Therefore,  
$$
\lim_{t \to \infty} f_i(x(t)) = \lim_{t \to \infty} \mathcal W_i(t) - \frac{1}{2}\|\dot x(t)\|^2 = \mathcal W_i^* := f_i^\infty \in \mathbb{R},
$$  
exists.  

(ii) From (i), for any $i \in \{1, \cdots, m\}$, we have  
$$
f_i(z) = f_i^\infty \le f_i(x(t)) + \frac{1}{2}\|\dot x(t)\|^2.
$$  
Rearranging the terms and taking the minimum over $i$ on both sides gets the result.
\end{proof}

\begin{lemma}\label{lem:Dynamical-system-lemma3} 
Suppose \cref{assum:Lj-muj,assum:levelset} hold. Let $ x: [1, +\infty) \to \mathbb{R}^n $ be a solution to the problem \cref{eq:CP} with $\alpha \in(0,3]$. Suppose $f_i$ is bounded below for $i=1,\cdots,m$. Let $z^*$ be any cluster point of $x$, then, 
$$
\lim_{t\to \infty }\mathcal E_{z^*}(t) =\mathcal E^*,
$$  
exists.
\end{lemma}
\begin{proof}
By \cref{lem:Dynamical-system-lemma2}, for any $t \ge s$ and $z\in \mathbb{R}^n$,  
$$
\mathcal E_z(t) - \mathcal E_z(s) \le \int_{s}^t \frac{d}{d\tau}\mathcal E_z(\tau)d\tau \le 0.
$$  
Hence,  
$$
\mathcal E_z(t) \le \mathcal E_z(1) \le u_0(x_0) + \frac{\alpha (\alpha +1)}{6}\|x_0 - z\|^2.
$$  
Setting $z = x(t)$ yields  
$$
\frac{1}{2}t^{\frac{2\alpha }{3}}\|\dot x(t)\|^2 \le  u_0(x_0)+\frac{\alpha (\alpha +1)}{6}\|x_0-x(t)\|^2\le u_0(x_0)+\frac{\alpha(\alpha +1)}{12}R_{x_0}^2:=C,
$$  
and by \cref{lem:Dynamical-system-point-convergence-1}, we have 
$$
\mathcal E_{z^*}(t) \ge -C.
$$  
Therefore, $\mathcal E_z(t)$ is monotonically non-increasing and bounded below, so $\lim_{t\to \infty }\mathcal E_{z^*}(t)$ exists.
\end{proof}
The following lemma is crucial to our main conclusion.
\begin{lemma}[\cite{boct2025fast}, Lemma A.4]\label{lem:continuous-limit} Let $a>0$ and $q:[t_0,+\infty )\to \R^n$ be a continuously differentiable function such that 
    \begin{equation}
    \lim_{t \to \infty}\left(q(t)+\frac{t}{a}\dot q(t)\right) = \ell\in \R.
    \end{equation}
    Then, it holds $\lim_{t\to \infty}q(t) = \ell$. 
\end{lemma} 
\begin{theorem}
    Suppose \cref{assum:Lj-muj,assum:levelset} hold. Let $x: [1, +\infty) \to \mathbb{R}^n$ be a solution to \cref{eq:CP} with $\alpha =3$. Then $x(t)$ converges to the set of weakly Pareto optimal solutions of \cref{eq:MOP}.
\end{theorem}
\begin{proof}
    Since the level set $\mathcal{L}(F,F(x_0))$ is bounded by \cref{assum:levelset}, it follows from  \cref{coro:Dynamical-system-corollary1} that $x$ is a bounded function. Let $z_1, z_2$ be any two cluster points of $x(t)$. Furthermore, according to \cref{lem:Dynamical-system-point-convergence-1} and the continuity of $f_i$, we have
    \begin{equation}\label{eq:fz1=fz2}
    f_i(z_1) = f_i(z_2) = f_i^\infty, \text{ for all } i = 1, \cdots, m. \end{equation} 
    Define  
    $$
    h_z(t) = \|x(t) - z\|^2.
    $$
    Since $x(t)$ is continuously differentiable, $h_z(t)$ is also continuously differentiable. Note the definition of $\mathcal{E}_z(t)$ with $\alpha =3$:  
    $$
    \begin{aligned}
    \mathcal{E}_z(t) &:= t^2 \min_{i=1,\cdots,m} \left( f_i(x(t)) - f_i(z) \right) + \frac{1}{2} t^2 \|\dot{x}(t)\|^2 + 2 \|x(t) - z\|^2 + 2t \langle \dot{x}(t), x(t) - z \rangle \\
    &= t^2 \min_{i=1,\cdots,m} \left( f_i(x(t)) - f_i(z) \right) + \frac{1}{2} t^2 \|\dot{x}(t)\|^2 + 2h_z(t) + t \dot{h}_z(t).
    \end{aligned}
    $$
    Thus, by \cref{eq:fz1=fz2}, we have
    $$
    \begin{aligned}
    \left( h_{z_1}(t) - h_{z_2}(t) \right) + \frac{t}{2} \left( \dot{h}_{z_1}(t) - \dot{h}_{z_2}(t) \right) = \frac{\mathcal{E}_{z_1}(t) - \mathcal{E}_{z_2}(t)}{2}.
    \end{aligned}
    $$
    Since $\lim_{t \to \infty} \left( \mathcal{E}_{z_1}(t) - \mathcal{E}_{z_2}(t) \right) = \mathcal{E}_{z_1}^* - \mathcal{E}_{z_2}^* := \ell \in \mathbb{R}$ by \cref{lem:Dynamical-system-lemma3}. Moreover,  By \cref{lem:continuous-limit}, we have  
    $$
    \lim_{t \to \infty} \left( h_{z_1}(t) - h_{z_2}(t) \right) = \frac{\ell}{2}.
    $$
    Since $z_1$ and $z_2$ are cluster points, there exist sequences $\{t_k\}, \{s_k\}$ such that $t_k, s_k \to \infty$, $x(t_k) \to z_1$, and $x(s_k) \to z_2$ as $k\to \infty $. Therefore,  
    $$
    \begin{aligned}
    \lim_{k \to \infty} \left( h_{z_1}(t_k) - h_{z_2}(t_k) \right) &= \lim_{k \to \infty} \left( \|x(t_k) - z_1\|^2 - \|x(t_k) - z_2\|^2 \right) = -\|z_1 - z_2\|^2, \\
    \lim_{k \to \infty} \left( h_{z_1}(s_k) - h_{z_2}(s_k) \right) &= \lim_{k \to \infty} \left( \|x(s_k) - z_1\|^2 - \|x(s_k) - z_2\|^2 \right) = \|z_1 - z_2\|^2.
    \end{aligned}
    $$
    Hence, $\|z_1 - z_2\|^2 = 0$, i.e., $z_1 = z_2$. This implies that there exists $x^* \in \mathbb{R}^n$ such that  
    $$
    \lim_{t \to \infty} x(t) = x^*.
    $$ 
    Furthermore, according to \cref{thm:weakpareto} (i) and (iii), we have  
    $$  
    0 \le u_0(x^*) \le \liminf_{t \to \infty} u_0(x(t)) = 0.  
    $$  
    Thus, $x^*$ is a weakly Pareto optimal solution of \cref{eq:MOP} by \cref{thm:weakpareto} (ii).
\end{proof}
\begin{remark}
For the \cref{eq:CP} with $0 < \alpha < 3$, it remains unsolved. The difficulty lies in the fact that for two distinct cluster points $z_1$ and $z_2$, the existence of the limit of $\mathcal{E}_{z_1}(t) - \mathcal{E}_{z_2}(t)$ does not guarantee the existence of $\lim_{t\to \infty }\left( h_{z_1}(t) - h_{z_2}(t) \right) + \frac{t}{2} \left( \dot{h}_{z_1}(t) - \dot{h}_{z_2}(t) \right)$.
\end{remark}
\section{Algorithm}\label{sec:algo}
In this section, we propose a multi-objective accelerated gradient algorithm with a generalized momentum parameter selection, a method initially studied by Tanabe et al. \cite{tanabe2022globally}. The specific iterative scheme is as follows:  
\begin{equation}\label{eq:algorithm-iterschme}
\left\{\begin{aligned}  
t_{k+1} &=\sqrt{t_k^2 - a t_k + b} + \frac{1}{2}, \\  
y_k &= x_k + \frac{t_k - 1}{t_{k+1}}(x_k - x_{k-1}), \\  
x_{k+1} &= y_k - s \cdot \proj_{C(y_k)}\left(\frac{1}{s}(y_k-x_k)\right).  
\end{aligned}\right.  \tag{MAG-GM}
\end{equation}
Here, $\{t_k\}$ is a sequence generated iteratively from the initial point $t_1 = 1$ using the above formula, with $a \in [0,1)$ and $b \in \left[\frac{a^2}{4}, \frac{1}{4}\right]$. It is worth noting that when $a = 0$ and $b = \frac{1}{4}$, the inertial parameter selection in this iterative scheme corresponds to that proposed by Tanabe et al. for the accelerated proximal gradient algorithm \cite{tanabe2023accelerated}. When $b = \frac{a^2}{4}$, this momentum parameter selection can be found in single-objective optimization references \cite{Attouch2014,su2016differential} and in multi-objective references \cite{sonntag2024fast,yin2025multiobjective}. In particular, when $a = b = 0$, the iterative scheme reduces to the multi-objective gradient-type algorithm proposed by Sonntag and Peitz \cite{sonntag2024fast}. Based on \cref{eq:algorithm-iterschme}, we propose a multi-objective accelerated algorithmic framework with generalized momentum factors.
\begin{algorithm}[H]
    \caption{Multiobjective accelerated gradient method with generalized momentum factor}\label{algo:Multiobjective}
    \begin{algorithmic}[1]
        \Require Choose $x_0=x_1\in \mathbb{R}^n$, $0<s\leq\frac{1}{L}$, $t_1 = 1$, $a\in[0,1)$ and $b\in[\frac{a^2}{4},\frac14]$. 
        \State Set $k=1$. 
        \While{$k\le k_{\max}$}
        \State Set $t_{k+1} =\sqrt{t_k^2-at_k+b}+\frac12$. 
        \State Set $y_k =x_k +\frac{t_{k}-1}{t_{k+1}}(x_k-x_{k-1})$. 
        \State Compute $\theta^k =(\theta_1^k,\cdots,\theta_m^k)\in \mathbb{R}^m$ by solving
        \begin{equation}
        \min_{\theta \in \Delta^m} \left\|s\left(\sum_{ i = 1}^m\theta_i \nabla f_i(y_k)\right)-(y_k-x_k)\right\|^2. 
        \end{equation}
        \State $x_{k+1}=y_k-s\sum_{i=1}^m\theta^{k}_i\nabla f_i(y_k)$. 
        \If {$\|x_{k+1}-y_k\|<\varepsilon$}
        \State \textbf{break}
        \Else
        \State Update $k\leftarrow k+1$. 
        \EndIf
        \EndWhile
    \end{algorithmic}
\end{algorithm}

When the termination condition given in line 7 of \cref{algo:Multiobjective} is satisfied, we have  
$$
\begin{aligned}  
\|\proj_{C(x_{k+1})}(0)\| &\le \|\proj_{C(x_{k+1})}(0) - \proj_{C(y_k)}(0)\| + \|\proj_{C(y_k)}(0)\| \\  
&\le \rho \|x_{k+1} - y_k\|^{\frac{1}{2}} + \frac{1}{s}\|x_{k+1} - y_k\| \\  
&\le \rho \cdot \varepsilon^{\frac{1}{2}} + \frac{1}{s} \cdot \varepsilon.  
\end{aligned}  
$$  
At this point, within the tolerance range, it can be considered that $\proj_{C(x_{k+1})}(0) = 0$, meaning $x_{k+1}$ is a critical point. In the subsequent content of this paper, the iterative sequence generated by \cref{algo:Multiobjective} is regarded as an infinite sequence generated by the iterative scheme \cref{eq:algorithm-iterschme}.

The following lemma presents some properties of the generalized momentum factor $\{t_k\}$.
\begin{lemma}[\cite{tanabe2022globally}, Lemma 6]\label{lem:tk-properties}
{Let $\{t_{k}\}$  be defined by lines 3 in \cref{algo:Multiobjective} with $a\in[0,1)$ and $b\in[\frac{a^{2}}4,\frac{1}{4}]$. Then, the following inequalities hold for all $k\geq 1$.}
\begin{itemize}
    \item[(i)] $t_{k+1}\geq t_{k}+\frac{1-a}{2}$ and $t_{k}\geq\frac{1-a}{2}k+\frac{1+a}{2}$;
    \item[(ii)] $t_{k+1}\leq t_{k}+\frac{1-a+\sqrt{4b-a^{2}}}{2}$ and $t_{k}\leq\frac{1-a+\sqrt{4b-a^{2}}}{2}(k-1)+1\leq k$;
    \item[(iii)] $t_{k}^{2}-t_{k+1}^{2}+t_{k+1}=at_{k}-b+\frac{1}{4}\geq at_{k}$;
    \item[(iv)] $0\leq\frac{t_k-1}{t_{k+1}}\leq\frac{k-1}{k+1/2}$;
    \item[(v)] $1-\left(\frac{t_k-1}{t_{k+1}}\right)^2\ge \frac{1}{t_k}$. 
\end{itemize}
\end{lemma}
\subsection{Properties of the generated sequence}
From this section onward, including subsequent sections, we define  
$$\sigma_k(z) = \min_{i=1,\cdots,m} \big( f_i(x_k) - f_i(z) \big).$$
\begin{lemma}\label{lem:inequal-sigma}
Suppose \cref{assum:Lj-muj} hold. Let $\{x_k\}$ and $\{y_k\}$ be the sequences  generated by \cref{algo:Multiobjective} with $a\in[0,1)$ and $b\in[\frac{a^{2}}{4},\frac{1}{4}]$. Then for any $k \ge 1$,
\begin{equation*}
\begin{aligned}  
\sigma_{k+1}(z) &\le -\frac{1}{s}\langle x_{k+1} - y_k, y_k - z \rangle - \frac{1}{2s}\|x_{k+1} - y_k\|^2, \\  
\sigma_{k+1}(z) - \sigma_k(z) &\le \max_{i=1,\cdots,m} \big(f_i(x_{k+1}) - f_i(x_k)\big) \\  
&\le -\frac{1}{s}\langle x_{k+1} - y_k, y_k - x_k \rangle - \frac{1}{2s}\|x_{k+1} - y_k\|^2.  
\end{aligned}
\end{equation*}  
\end{lemma}
\begin{proof}
The proof is similar to that of Lemmas 6.4 and 6.5 in \cite{sonntag2024fast}, and is thus omitted here.
\end{proof}
\begin{corollary}\label{coro:inequal-sigma}
Suppose \cref{assum:Lj-muj} hold. Let $\{x_k\}$ and $\{t_k\}$ be the sequences  generated by \cref{algo:Multiobjective} with $a\in[0,1)$ and $b\in[\frac{a^{2}}{4},\frac{1}{4}]$. For any $1\le k_1 \le k_2$, we have
\begin{equation*}
\sigma_{k_2}(z) - \sigma_{k_1}(z) \le \frac{1}{2s} \left[ \|x_{k_1} - x_{k_1-1}\|^2 - \|x_{k_2} - x_{k_2-1}\|^2 \right] - \frac{1}{2s} \sum_{k=k_1}^{k_2-1} \frac{1}{t_k} \|x_k - x_{k-1}\|^2.
\end{equation*}
\end{corollary}
\begin{proof}
According to \cref{lem:inequal-sigma}, we have
\begin{equation}\label{eq:sigmamiussigam}
\begin{aligned}
&~~~~~\sigma_{k+1}(z) - \sigma_k(z) \le \max_{i=1,\cdots,m} \left( f_i(x_{k+1}) - f_i(x_k) \right) \\
&\le -\frac{1}{s} \langle x_{k+1} - y_k, y_k - x_k \rangle - \frac{1}{2s} \|x_{k+1} - y_k\|^2 \\
&= \frac{1}{2s} \left[ \|y_k - x_k\|^2 - \|x_{k+1} - x_k\|^2 \right] \\
&= \frac{1}{2s} \left[ \left( \frac{t_k - 1}{t_{k+1}} \right)^2 \|x_k - x_{k-1}\|^2 - \|x_{k+1} - x_k\|^2 \right] \\
&= \frac{1}{2s} \left[ \|x_k - x_{k-1}\|^2 - \|x_{k+1} - x_k\|^2 \right] + \frac{1}{2s} \left( \left( \frac{t_k - 1}{t_{k+1}} \right)^2 - 1 \right) \|x_k - x_{k-1}\|^2 \\
&\le \frac{1}{2s} \left[ \|x_k - x_{k-1}\|^2 - \|x_{k+1} - x_k\|^2 \right] - \frac{1}{2s} \frac{1}{t_k} \|x_k - x_{k-1}\|^2.
\end{aligned}
\end{equation}
where the last inequality follows from \cref{lem:tk-properties}. Finally, summing the above inequality from $k = k_1$ to $k_2 - 1$, we obtain the conclusion.
\end{proof}
A key property of the multi-objective acceleration algorithm is that the generated sequence lies within a level set. For the sequence $\{x_k\}$ generated by \cref{algo:Multiobjective}, we define a set of auxiliary sequences $\{\mathcal{W}_k^i\}_{k=1}^\infty$, $i=1,\cdots,m$, to derive similar conclusions.
\begin{equation}\label{eq:W_ik}
\mathcal W_k^i:=f_i(x_k)+\frac{1}{2s}\|x_k-x_{k-1}\|^2.
\end{equation}
\begin{lemma}\label{lem:aux1}
Suppose \cref{assum:Lj-muj} hold. Let $\{x_k\}$ and $\{t_k\}$ be the sequences  generated by \cref{algo:Multiobjective} with $a\in[0,1)$ and $b\in[\frac{a^{2}}{4},\frac{1}{4}]$. Let $\{\mathcal W_k^i\}$ be defined in \cref{eq:W_ik}. Then, 
\begin{itemize}
    \item[(i)] $\{\mathcal{W}_k^i\}$ is monotonically non-increasing for $i=1,\cdots,m$;
    \item[(ii)]  Assuming $f_i $ is bounded below for $i=1,\cdots,m$, then $ \lim_{k \to \infty} \mathcal{W}_k^i $ exists, and
    $$\sum_{k=1}^\infty \frac{1}{t_k} \|x_k - x_{k-1}\|^2 < +\infty.$$
\end{itemize}
\end{lemma}
\begin{proof}
(i) From \cref{eq:sigmamiussigam}, we obtain  
$$f_i(x_{k+1}) - f_i(x_k) \le \frac{1}{2s} \left[ \|x_k - x_{k-1}\|^2 - \|x_{k+1} - x_k\|^2 \right] - \frac{1}{2s} \frac{1}{t_k} \|x_k - x_{k-1}\|^2.$$  
Thus, for any $k \ge 1$,  
\begin{equation}\label{eq:W_ikt_k}\mathcal{W}_{k+1}^i + \frac{1}{2s} \frac{1}{t_k} \|x_k - x_{k-1}\|^2 \le \mathcal{W}_k^i.\end{equation}  
Therefore, $\{\mathcal{W}_k^i\}$ is monotonically non-increasing.

(ii) Since $ f_i $ has a lower bound $ \inf f_i > -\infty $, it follows that $ \mathcal{W}_k^i \ge f_i(x_k) \ge \inf f_i $ for all $ k \ge 1 $. Therefore, $ \lim_{k \to \infty} \mathcal{W}_k^i $ exists. Furthermore, according to  \cref{eq:W_ikt_k}, we have  
$$\sum_{k=1}^n \frac{1}{t_k} \|x_k - x_{k-1}\|^2 \le \sum_{k=1}^n \left( \mathcal{W}_k^i - \mathcal{W}_{k+1}^i \right) = \mathcal{W}_1^i - \mathcal{W}_{n+1}^i \le f_i(x_0) - \inf f_i.$$  
Taking the limit as $ n \to \infty $, the conclusion follows.
\end{proof}
\begin{corollary}\label{coro:levelset}
Suppose \cref{assum:Lj-muj} hold. Let $\{x_k\}$  be the sequence  generated by \cref{algo:Multiobjective} with $a\in[0,1)$ and $b\in[\frac{a^{2}}{4},\frac{1}{4}]$. Then, we have
$$f_i(x_k)\le f_i(x_0),$$
for $i=1,\cdots,m$. 
\end{corollary}
\begin{proof}
According to \cref{lem:aux1} (i), $\{\mathcal{W}_k^i\}$ is monotonically non-increasing for $i=1,\cdots,m$. Therefore, for any $k \ge 1$, we have  
$$f_i(x_k) \le \mathcal{W}_k^i \le \mathcal{W}_1^i = f_i(x_1) + \frac{1}{2s} \|x_1 - x_0\|^2 = f_i(x_0).$$  
This completes the proof.
\end{proof}
\subsection{Convergence rate}
For any $z\in \mathbb{R}^n$, let the auxiliary sequence (discrete Lyapunov function) be defined as  
\begin{equation} \label{eq:Lyapunov-sequence}\mathcal{E}_k(z) = t_k^2 \sigma_k(z) + \frac{1}{2s} \|\eta_k - z\|^2,\end{equation}  
where $\eta_k = x_k + (t_k - 1)(x_k - x_{k-1})$. Then,  
\begin{equation}\label{eq:eta-eta}
\begin{aligned}
\eta_{k+1} - \eta_k &= x_{k+1} + (t_{k+1} - 1)(x_{k+1} - x_k) - x_k - (t_k - 1)(x_k - x_{k-1}) \\
&= t_{k+1}(x_{k+1} - x_k) - (t_k - 1)(x_k - x_{k-1}) \\
&= t_{k+1}(x_{k+1} - y_k), \\
\eta_{k+1} + \eta_k &= \eta_{k+1} - \eta_k + 2\eta_k \\
&= t_{k+1}(x_{k+1} - y_k) + 2x_k + 2(t_k - 1)(x_k - x_{k-1}) \\
&= t_{k+1}(x_{k+1} - y_k) + 2x_k + 2t_{k+1}(y_k - x_k),
\end{aligned}
\end{equation}  
Additionally, we denote  
\begin{equation}
\label{eq:zeta}
\zeta_k(a, b) = a t_k - b + \frac{1}{4},
\end{equation}
where $a\in[0,1)$ and $b\in[\frac{a^2}{4},\frac14]$. 

In single-objective optimization, the discrete Lyapunov function is always monotonically non-increasing. However, this conclusion does not hold in multi-objective optimization. As shown in the following lemma, since $\sigma_k(z)$ is not non-negative, we can only obtain a slightly weaker result for $\{\mathcal{E}_k(z)\}$ than monotonic non-increasing.
\begin{lemma}\label{lem:Lyaounov1}
Suppose \cref{assum:Lj-muj} hold. Let $\{x_k\}$ and $\{t_k\}$ be the sequences  generated by \cref{algo:Multiobjective} with $a\in[0,1)$ and $b\in[\frac{a^{2}}{4},\frac{1}{4}]$. For  $ z \in \mathbb{R}^n $, let $\{\mathcal{E}_k(z)\}$ and $\{\zeta_k(a,b)\}$ be defined as in \cref{eq:Lyapunov-sequence} and \cref{eq:zeta}, respectively. Then for any $ k \ge 1 $, the following inequality holds:  
$$\mathcal{E}_{k+1}(z) - \mathcal{E}_k(z) + \zeta_k(a,b) \sigma_k(z) \le 0.$$
\end{lemma}
\begin{proof}
For ease of discussion, we denote $\zeta_k = \zeta_k(a,b)$, and from \cref{lem:tk-properties}, we have $t_{k+1} = t_{k+1}^2 - t_k^2 + \zeta_k$. Then,  
\begin{equation}\label{eq:lemma-convergence-1}
\begin{aligned}
&~~~~~t_{k+1}^2\sigma_{k+1}(z) - t_k^2\sigma_k(z) + \zeta_k\sigma_k(z) \\&= (t_{k+1}^2 - t_k^2 + \zeta_k)\sigma_{k+1}(z) + (t_k^2 - \zeta_k)(\sigma_{k+1}(z) - \sigma_k(z)) \\
&= t_{k+1}\left[\sigma_{k+1}(z) + (t_{k+1} - 1)(\sigma_{k+1}(z) - \sigma_k(z))\right].
\end{aligned}
\end{equation}
Moreover, by \cref{lem:inequal-sigma}, we have
\begin{equation}
\label{eq:lemma-convergence-2}
\begin{aligned}
&~~~~~\sigma_{k+1}(z) + (t_{k+1} - 1)(\sigma_{k+1}(z) - \sigma_k(z)) \\
&\le -\frac{1}{s}\langle x_{k+1} - y_k, y_k - z\rangle - \frac{1}{2s}\|x_{k+1} - y_k\|^2 \\
& \qquad -\frac{1}{s}\langle x_{k+1} - y_k, (t_{k+1} - 1)(y_k - x_k)\rangle - \frac{t_{k+1} - 1}{2s}\|x_{k+1} - y_k\|^2 \\
 &= -\frac{1}{2s}\langle x_{k+1} - y_k, 2t_{k+1}y_k - 2(t_{k+1} - 1)x_k - 2z + t_{k+1}(x_{k+1} - y_k)\rangle \\
&= -\frac{1}{2s}\langle x_{k+1} - y_k, 2t_{k+1}(y_k - x_k) + 2x_k + t_{k+1}(x_{k+1} - y_k) - 2z\rangle \\
&= -\frac{1}{2s}\langle x_{k+1} - y_k, \eta_{k+1} + \eta_k - 2z\rangle.
\end{aligned}
\end{equation}
where the last equality is derived from \cref{eq:eta-eta}. Additionally, we have  
\begin{equation}\label{eq:eta-norm}\begin{aligned}
\|\eta_{k+1} - z\|^2 - \|\eta_k - z\|^2 &= \|\eta_{k+1} - \eta_k\|^2 + 2\langle \eta_{k+1} - \eta_k, \eta_k - z \rangle \\
&= \langle t_{k+1}(x_{k+1} - y_k), \eta_{k+1} + \eta_k - 2z \rangle.
\end{aligned}\end{equation}  
Combining \cref{eq:lemma-convergence-1,eq:lemma-convergence-2,eq:eta-norm}, we obtain  
$$
t_{k+1}^2\sigma_{k+1}(z) - t_k^2\sigma_k(z) + \zeta_k\sigma_k(z) \le -\frac{1}{2s}\left[\|\eta_{k+1} - z\|^2 - \|\eta_k - z\|^2\right].
$$
By the definition of $\mathcal{E}_k(z)$, rearranging the above inequality yields the conclusion. 
\end{proof}
\begin{lemma}
\label{lem:Lyapunov2}
Assuming the conditions stated in \cref{lem:Lyaounov1} hold, then for any $ k \ge 1 $ and $z\in \mathbb{R}^n$,  
$$
\begin{aligned}
\mathcal{E}_{k+1}(z) +  \sigma_{k+1}(z) \sum_{p=1}^{k} \zeta_p(a,b)  + \frac{1}{2s} \sum_{p=1}^k \left( \frac{a^2}{2}(p-1) + Q(a,b) \right) \|x_p - x_{p-1}\|^2 & \le \mathcal{E}_1(z).
\end{aligned}
$$  
where $ Q(a,b) = \frac{1}{2(1 - a)} \left( \frac{1}{4} - b \right)^2 $.
\end{lemma}
\begin{proof}
According to \cref{coro:inequal-sigma}, for any $ p $ and $ k \ge p $, we have  
$$\sigma_{k+1}(z) - \frac{1}{2s} \left[ \|x_p - x_{p-1}\|^2 - \|x_{k+1} - x_k\|^2 \right] + \frac{1}{2s} \sum_{l=p}^{k} \frac{1}{t_l} \|x_l - x_{l-1}\|^2 \le \sigma_p(z).$$  
Further, by \cref{lem:Lyaounov1}, we obtain  
$$\begin{aligned}\mathcal{E}_{p+1}(z) - \mathcal{E}_p(z) + \zeta_p(a,b) \sigma_{k+1}(z)& + \frac{\zeta_p(a,b)}{2s} \sum_{l=p}^{k} \frac{1}{t_l} \|x_l - x_{l-1}\|^2 \\&- \frac{\zeta_p(a,b)}{2s} \|x_p - x_{p-1}\|^2 \le 0.\end{aligned}$$  
Summing the above inequality from $ p = 1 $ to $ k $, we obtain  
\begin{equation}\label{eq:pp}
\begin{aligned}
\mathcal{E}_{k+1}(z) - \mathcal{E}_1(z) + \left( \sum_{p=1}^{k} \zeta_p(a,b) \right) \sigma_{k+1}(z) & + \sum_{p=1}^{k} \frac{\zeta_p(a,b)}{2s} \sum_{l=p}^{k} \frac{1}{t_l} \|x_l - x_{l-1}\|^2 \\
& - \sum_{p=1}^{k} \frac{\zeta_p(a,b)}{2s} \|x_p - x_{p-1}\|^2 \le 0.
\end{aligned}
\end{equation}  
Note that  
$$
L_1 := \sum_{p=1}^{k} \frac{\zeta_p(a,b)}{2s} \sum_{l=p}^{k} \frac{1}{t_l} \|x_l - x_{l-1}\|^2 = \frac{1}{2s} \sum_{p=1}^{k} \left( \sum_{l=1}^{p} \zeta_l(a,b) \right) \frac{1}{t_p} \|x_p - x_{p-1}\|^2.
$$  
Thus,  
$$
\begin{aligned}
L_1 - \sum_{p=1}^{k} \frac{\zeta_p(a,b)}{2s} \|x_p - x_{p-1}\|^2 = \frac{1}{2s} \sum_{p=1}^{k} \left[ \left( \sum_{l=1}^{p} \zeta_l(a,b) \right) \frac{1}{t_p} - \zeta_p(a,b) \right] \|x_p - x_{p-1}\|^2.
\end{aligned}
$$  
Since  
$$
\begin{aligned}
\left( \sum_{l=1}^{p} \zeta_l(a,b) \right) \frac{1}{t_p} - \zeta_p(a,b) = \frac{1}{t_p} \left[ a \left( \sum_{l=1}^{p-1} t_l - t_p^2 + t_p \right) + \left( \frac{1}{4} - b \right) (p - t_p) \right],
\end{aligned}
$$  
and $ t_1 = 1 $, by \cref{lem:tk-properties} (iii),  
$$
\begin{aligned}
-t_p^2 + t_p &= \sum_{l=1}^{p-1} \left[ -t_{l+1}^2 + t_{l+1} - (-t_l^2 + t_l) \right] = \sum_{l=1}^{p-1} \left( \frac{1}{4} + (a - 1)t_l - b \right) \\
&= -(1 - a) \sum_{l=1}^{p-1} t_l + \left( \frac{1}{4} - b \right)(p - 1),
\end{aligned}
$$  
we obtain  
$$
\sum_{l=1}^{p-1} t_l = \frac{t_p^2 - t_p}{1 - a} + \left( \frac{1}{4} - b \right) \frac{p - 1}{1 - a}.
$$  
Therefore,  
$$
\begin{aligned}
&~~~~~L_1 - \sum_{p=1}^{k} \frac{\zeta_p(a,b)}{2s} \|x_p - x_{p-1}\|^2 \\&= \frac{1}{2s(1 - a)} \sum_{p=1}^{k} \left[ a^2(t_p - 1) + \left( \frac{1}{4} - b \right) \frac{p - t_p + a(t_p - 1)}{t_p} \right] \|x_p - x_{p-1}\|^2.
\end{aligned}
$$  
Since $ t_p \le \frac{1 - a + \sqrt{4b - a^2}}{2}(p - 1) + 1 $ (\cref{lem:tk-properties}),  
$$
\begin{aligned}
p - t_p &\ge p - \frac{1 - a + \sqrt{4b - a^2}}{2}(p - 1) - 1 \\
&= \frac{1 + a - \sqrt{4b - a^2}}{2}(p - 1) \ge \frac{1 - \sqrt{4b - a^2}}{2}(p - 1).
\end{aligned}
$$  
Moreover, since $ t_k \ge 1 $ and $ b \in (a^2/4, 1/4] $, for any $ p \ge 2 $,  
$$
\begin{aligned}
& ~~~~~\frac{1}{1 - a} \left[ a^2(t_p - 1) + \left( \frac{1}{4} - b \right) \frac{p - t_p + a(t_p - 1)}{t_p} \right] \\
&=  \frac{a^2}{1 - a}(t_p - 1) + \frac{1}{1 - a} \left( \frac{1}{4} - b \right) \frac{ \frac{1 - \sqrt{4b - a^2}}{2}(p - 1) }{t_p} + \frac{1}{1 - a} \left( \frac{1}{4} - b \right) \frac{a(t_p - 1)}{t_p} \\
&\ge  \frac{a^2}{1 - a} \left( \frac{1 - a}{2}p + \frac{a - 1}{2} \right) + \frac{1}{1 - a} \left( \frac{1}{4} - b \right) \frac{ \frac{1 - \sqrt{4b - a^2}}{2}(p - 1) }{p} \\
&\ge  \frac{a^2}{2}(p - 1) + \frac{1}{2(1 - a)} \left( \frac{1}{4} - b \right) (1 - \sqrt{4b - a^2}) \left( 1 - \frac{1}{p} \right) \\
&\ge\frac{a^2}{2}(p - 1) + \frac{1}{2(1 - a)} \left( \frac{1}{4} - b \right)^2.
\end{aligned}
$$  
Define $ Q(a,b) = \frac{1}{2(1 - a)} \left( \frac{1}{4} - b \right)^2$. Since $ x_1 = x_0 $, we have  
$$
L_1 - \sum_{p=1}^{k} \frac{\zeta_p(a,b)}{2s} \|x_p - x_{p-1}\|^2 \ge \frac{1}{2s} \sum_{p=1}^{k} \left( \frac{a^2}{2}(p - 1) + Q(a,b) \right) \|x_p - x_{p-1}\|^2.
$$  
Combining this with \cref{eq:pp}, the conclusion follows.
\end{proof}

The following lemma presents the convergence rate results of the algorithm.
\begin{theorem}\label{thm:convergence-rate-and-sequence}
Suppose \cref{assum:Lj-muj,assum:levelset} hold. Let $\{x_k\}$ and $\{t_k\}$ be the sequences  generated by \cref{algo:Multiobjective} with $a\in[0,1)$ and $b\in[\frac{a^{2}}{4},\frac{1}{4}]$. Then, 
\begin{itemize}
    \item[(i)] $u_0(x_k ) =O(1/k^2)$;
    \item[(ii)] $\sum_{k=1}^\infty \left(ak-b+\frac14\right)\|x_{k}-x_{k-1}\|^2<+\infty$. 
\end{itemize}
\end{theorem}
\begin{proof}
(i) From the definition of $\mathcal{E}_k(z)$ and \cref{lem:Lyapunov2}, we obtain  
\begin{equation}\label{eq:point-convergnece-1} 
\begin{aligned} 
\left(t_k^2 + \sum_{p=1}^{k} \zeta_p(a,b)\right) \sigma_{k+1}(z) + \frac{1}{2s} \sum_{p=1}^k \left( \frac{a^2}{2}(p-1) + Q(a,b) \right) \|x_p - x_{p-1}\|^2 &\\\le \mathcal{E}_1(z).&
\end{aligned}
\end{equation} 
By \cref{coro:levelset}, $\{x_k\} \subseteq \mathcal{L}(F, F(x_0))$, we have  
$$
\sup_{z \in \mathbb{R}^n} \sigma_{k+1}(z) = \sup_{z \in \mathcal{L}(F, F(x_0))} \sigma_{k+1}(z).
$$  
Moreover,  
$$
\begin{aligned} 
\sup_{z \in \mathcal{L}(F, F(x_0))} \mathcal{E}_1(z) \le u_0(x_0) + \sup_{z \in \mathcal{L}(F, F(x_0))} \frac{1}{2s} \|x_0 - z\|^2 \le u_0(x_0) + \frac{1}{s} R_{x_0}^2.
\end{aligned} 
$$  
Taking the supremum over $z \in \mathcal{L}(F, F(x_0))$ in inequality \cref{eq:point-convergnece-1}, we derive  
$$
\begin{aligned} 
\left(t_{k+1}^2 + \sum_{p=1}^{k} \zeta_p(a,b)\right) u_0(x_{k+1}) + \frac{1}{2s} \sum_{p=1}^k \left( \frac{a^2}{2}(p-1) + Q(a,b) \right) \|x_p - x_{p-1}\|^2 &\\\le u_0(x_0) + \frac{R_{x_0}^2}{s}.&
\end{aligned} 
$$  
Thus, for any $k \ge 1$,  
$$
u_0(x_{k+1}) \le \frac{u_0(x_0) + \frac{R_{x_0}^2}{s}}{t_{k+1}^2 + \sum_{p=1}^{k+1} \zeta_p(a,b)} \le \frac{s u_0(x_0) + R_{x_0}^2}{s(1-a)^2 (k+1)^2}.
$$  
The second inequality follows from \cref{lem:tk-properties}.  

(ii) Similarly, for any $k \ge 1$,  
$$
\sum_{p=1}^k \left( \frac{a^2}{2}(p-1) + Q(a,b) \right) \|x_p - x_{p-1}\|^2 \le s u_0(x_0) + R_{x_0}^2.
$$  
When $a = 0$ and $b = \frac{1}{4}$, $\frac{a^2}{2}(p-1) + Q(a,b) = 0$, so the conclusion holds trivially. When $a > 0$, for any $k \ge 1$,  
$$
\sum_{p=1}^k p \|x_p - x_{p-1}\|^2 \le 2 \sum_{p=1}^k (p-1) \|x_p - x_{p-1}\|^2 \le \frac{4}{a^2} (s u_0(x_0) + R_{x_0}^2).
$$  
When $b < \frac{1}{4}$, we have $Q(a,b) > 0$, and for any $k \ge 1$,  
$$
\sum_{p=1}^k \|x_p - x_{p-1}\|^2 < \frac{s u_0(x_0) + R_{x_0}^2}{Q(a,b)}.
$$  
Let 
$$
C_{a,b}:=\left\{\begin{aligned}
&0,&&a = 0, b =\frac14;\\
&\frac{4}{a}(su_0(x_0)+R_{x_0}),&&a\in(0,1),b=\frac{1}{4};\\
&\left(\frac14-b\right)\frac{su_0(x_0)+R_{x_0}}{Q(a,b)}, && a=0,b\in\left(0,\frac14\right);\\
&\frac{4}{a}(su_0(x_0)+R_{x_0})+\left(\frac14-b\right)\frac{su_0(x_0)+R_{x_0}}{Q(a,b)}&& a\in(0,1),b\in\left(0,\frac14\right). 
\end{aligned}\right. 
$$
Therefore, for any $k \ge 1$,  
$$
\sum_{p=1}^k \left( a p - b + \frac{1}{4} \right) \|x_p - x_{p-1}\|^2 \le C_{a,b}. 
$$  
Taking the limit as $k \to \infty$, the result follows. 
\end{proof}
\subsection{Point convergence }
In this subsection, under the condition that the objective functions are bounded below, we present the convergence results of the sequence generated by \cref{algo:Multiobjective}. The following lemma shows that $\{f_i(x_k)\}$ is convergent. Moreover, for the cluster points $z^*$ of $\{x_k\}$, $\{\sigma_k(z^*)\}$ is bounded below.
\begin{lemma}\label{lem:point-convergence-1}
Suppose \cref{assum:Lj-muj,assum:levelset} hold. Let $\{x_k\}$ and $\{t_k\}$ be the sequences  generated by \cref{algo:Multiobjective} with $a\in[0,1)$ and $b\in[\frac{a^{2}}{4},\frac{1}{4}]$. Suppose $f_i$ is bounded below for $i=1,\cdots,m$, then, 
\begin{itemize}
    \item[(i)] $\lim_{k\to \infty }f_i(x_k)=f_i^\infty$ exists;
    \item[(ii)]  Let $z^*$ be any cluster point of $\{x_k\}$. Then for any $k \ge 1$,  
    $$-\sigma_{k}(z^*) \le \frac{1}{2s} \|x_k - x_{k-1}\|^2.$$
\end{itemize}
\end{lemma}
\begin{proof}
(i) By \cref{coro:inequal-sigma}, for any $ z \in \mathbb{R}^n $, we have  
\begin{equation}\label{eq:point-convergence-sigma-1}
\begin{aligned}
\sigma_{k+1}(z) + \frac{1}{2s} \|x_{k+1} - x_k\|^2 + \frac{1}{2s} \frac{1}{t_k} \|x_k - x_{k-1}\|^2 \le \sigma_k(z) + \frac{1}{2s} \|x_k - x_{k-1}\|^2.
\end{aligned}
\end{equation}
Let $ \mathcal{W}_k := u_0(x_k) + \frac{1}{2s} \|x_k - x_{k-1}\|^2 $. Taking the supremum over $ z \in \mathbb{R}^n $ in \cref{eq:point-convergence-sigma-1}, we obtain  
$$
0 \le \mathcal{W}_{k+1} \le \mathcal{W}_k.
$$  
Thus, $ \lim_{k \to \infty} \mathcal{W}_k = \mathcal{W}^* $ exists. By \cref{thm:convergence-rate-and-sequence}, $ \lim_{k \to \infty} u_0(x_k) = 0 $, so  
\begin{equation}\label{eq:point-convergence-xklimitexist}
\lim_{k \to \infty} \|x_k - x_{k-1}\|^2 = 2s \lim_{k \to \infty} (\mathcal{W}_k - u_0(x_k)) = 2s \mathcal{W}^*\in \mathbb{R}_{+}.
\end{equation} 
Since $ f_i $ is bounded below, by \cref{lem:tk-properties}(ii) and \cref{lem:aux1}, we have  
\begin{equation}\label{eq:point-convergence-summable} 
\sum_{k=1}^{\infty} \frac{1}{k} \|x_k - x_{k-1}\|^2 \le \sum_{k=1}^{\infty} \frac{1}{t_k} \|x_k - x_{k-1}\|^2 < +\infty.
\end{equation}
From \cref{eq:point-convergence-xklimitexist,eq:point-convergence-summable}, it follows that  
$$
\lim_{k \to \infty} \|x_k - x_{k-1}\|^2 = 0.
$$  
Since $ \{\mathcal{W}_k^i\} $ has a limit and is monotonically non-increasing (\cref{lem:aux1}),  
$$
\lim_{k \to \infty} f_i(x_k) = \lim_{k \to \infty} \left( \mathcal{W}_k^i - \frac{1}{2s} \|x_k - x_{k-1}\|^2 \right) = \inf_{k \ge 1} \mathcal{W}_k^i := f_i^\infty \in \mathbb{R}.
$$  

(ii) Since $ z^* $ is a cluster point of $ \{x_k\} $, $ f_i^\infty = f_i(z^*) $ holds for all $ i = 1, \cdots, m $. By the monotonic non-increasing property of $ \{\mathcal{W}_k^i\} $, for any $ k \ge 1 $,  
$$
f_i(x_k) + \frac{1}{2s} \|x_k - x_{k-1}\|^2 = \mathcal{W}_k^i \ge \inf_{k \ge 1} \mathcal{W}_k^i = f_i^\infty = f_i(z^*),
$$  
holds for all $ i = 1, \cdots, m $. Furthermore,  
$$
\sigma_k(z^*)+ \frac{1}{2s} \|x_k - x_{k-1}\|^2 = \min_{i=1,\cdots,m} (f_i(x_k) - f_i(z^*)) + \frac{1}{2s} \|x_k - x_{k-1}\|^2 \ge 0.
$$  
This completes the proof.
\end{proof}
\begin{lemma}\label{lem:Lyapunov-limit-exists}
Suppose \cref{assum:Lj-muj,assum:levelset} hold. Let $\{x_k\}$ and $\{t_k\}$ be the sequences  generated by \cref{algo:Multiobjective} with $a\in[0,1)$ and $b\in[\frac{a^{2}}{4},\frac{1}{4}]$.  For  $ z \in \mathbb{R}^n $, let $\{\mathcal{E}_k(z)\}$ be defined as in \cref{eq:Lyapunov-sequence}. Suppose $f_i$ is bounded below for $i=1,\cdots,m$, then for any cluster point $z^*$ of $\{x_k\}$, we have 
$$\lim_{k \to \infty}\mathcal{E}_k(z^*) =\mathcal E^*_{z^*}, $$exists. 
\end{lemma}
\begin{proof}
According to \cref{lem:Lyapunov2}, taking $ z = x_{k+1} $, we obtain  
$$
\frac{(t_{k+1} - 1)^2}{2s} \|x_{k+1} - x_k\|^2 = \mathcal{E}_{k+1}(x_{k+1}) \le \mathcal{E}_1(x_{k+1}) \le u_0(x_0) + \frac{1}{2s} \|x_1 - x_{k+1}\|^2.
$$  
Since $ \{x_k\} \subseteq \mathcal{L}(F, F(x_0)) $ and this level set is bounded by Assumption 2, we have  
$$
\frac{(t_{k+1} - 1)^2}{2s} \|x_{k+1} - x_k\|^2 \le u_0(x_0) + \frac{R^2}{s} := C.
$$  
By \cref{lem:point-convergence-1}, we have  
\begin{equation}\label{eq:point-convergence-lyapunov-bounded}
\begin{aligned}
\mathcal{E}_{k+1}(z^*) &\ge t_{k+1}^2 \sigma_{k+1}(z^*) \ge \frac{t_{k+1}^2}{2s} \|x_{k+1} - x_k\|^2 \\
&\ge \left( \frac{t_{k+1}}{t_{k+1} - 1} \right)^2 \frac{(t_{k+1} - 1)^2}{2s} \|x_{k+1} - x_k\|^2 \\
&\ge \left( 1+\frac1k \right)^2 \frac{(t_{k+1} - 1)^2}{2s} \|x_{k+1} - x_k\|^2 \ge -C,
\end{aligned}
\end{equation}  
where the second last inequality holds based on \cref{lem:tk-properties}. Moreover, according to \cref{lem:Lyaounov1}, the following inequality holds,  
$$
\mathcal{E}_{k+1}(z^*) - \mathcal{E}_k(z^*) + \zeta_k(a,b) \sigma_k(z^*) \le 0.
$$  
 Note that $ \zeta_k(a,b) = a t_k - b + \frac{1}{4} \le a k - b + \frac{1}{4} $, we have
\begin{equation}\label{eq:point-convergence-lyapunovlimitexist-1}
\begin{aligned} 
\mathcal{E}_{k+1}(z^*) - \mathcal{E}_k(z^*)& \le -\zeta_k(a,b) \sigma_k(z^*) \le \frac{\zeta_k(a,b)}{2s} \|x_k - x_{k-1}\|^2\\& \le \frac{1}{2s} \left( a k - b + \frac{1}{4} \right) \|x_k - x_{k-1}\|^2.
\end{aligned}
\end{equation}  
By \cref{thm:convergence-rate-and-sequence}, the  right-hand side of \cref{eq:point-convergence-lyapunovlimitexist-1} is summable, and $ \mathcal{E}_k(z^*) $ is bounded below by \cref{eq:point-convergence-lyapunov-bounded}. Therefore, the limit of $ \{\mathcal{E}_k(z^*)\} $ exists.
\end{proof}
\begin{lemma}\cite{boct2025acceleratingDiagonalMethods}\label{lem:desecert}
    Let $\{\varphi_k\}$, $\{q_k\}$ be sequences of real and positive numbers respectively, and such that $\sum_{k=1}^{\infty} \frac{1}{q_k} = +\infty$ and, 
    \begin{equation}
    \lim_{k \to \infty} (\varphi_{k+1}+q_k(\varphi_{k+1}-\varphi_k)) = \ell  \in \mathbb{R}.  
    \end{equation}
    Then, $ \lim_{k \to \infty} \varphi_k = \ell $.
\end{lemma}
\begin{theorem}\label{thm:point-convergence-result}
Suppose \cref{assum:Lj-muj,assum:levelset} hold. Let $\{x_k\}$ and $\{t_k\}$ be the sequences  generated by \cref{algo:Multiobjective} with $a\in[0,1)$ and $b\in[\frac{a^{2}}{4},\frac{1}{4}]$.  For  $ z \in \mathbb{R}^n $, let $\{\mathcal{E}_k(z)\}$ be defined as in \cref{eq:Lyapunov-sequence}. Suppose $f_i$ is bounded below for $i=1,\cdots,m$, then there exists $x^*$  such that $ \lim_{k \to \infty} x_k = x^* $, and $ x^*$ is a weakly Pareto optimal solution of \cref{eq:MOP}.
\end{theorem}
\begin{proof}
Let $ z_1 $ and $ z_2 $ be any two cluster points of the sequence $ \{x_k\} $. Define the sequence  
$$
h_k := \|x_k - z_1\|^2 - \|x_k - z_2\|^2.
$$  
Based on \cref{lem:point-convergence-1} and the continuity of $f_i$, we obtain $f_i(z_1)=f_i(z_2)$ for all $ i = 1, \cdots, m $, and hence $ \sigma_{k+1}(z_1) = \sigma_{k+1}(z_2) := \sigma_\infty $. Furthermore, according to the definition of $\mathcal E_k(z)$, the following equation holds:  
\begin{equation}\label{eq:point-convergence-result-1}
\begin{aligned}
\mathcal{E}_{k+1}(z_j) &= t_{k+1}^2 \sigma_\infty + \frac{1}{2s} \left\| (t_{k+1} - 1)(x_{k+1} - x_k) + x_{k+1} - z_j \right\|^2 \\
&= t_{k+1}^2 \sigma_\infty + \frac{(t_{k+1} - 1)^2}{2s} \|x_{k+1} - x_k\|^2 \\
&\quad + \frac{1}{2s} \left[ \|x_{k+1} - z_j\|^2 + 2(t_{k+1} - 1) \langle x_{k+1} - x_k, x_{k+1} - z_j \rangle \right], \quad j = 1, 2.
\end{aligned}
\end{equation}  
Observation shows that the first two terms on the right-hand side of \cref{eq:point-convergence-result-1} are the same for $\mathcal E_{k+1}(z_1)$ and $\mathcal E_{k+1}(z_2)$, so
\begin{equation}\label{eq:point-convergence-result-2}
\begin{aligned}
2s \left( \mathcal{E}_{k+1}(z_1) - \mathcal{E}_{k+1}(z_2) \right) &= \|x_{k+1} - z_1\|^2 - \|x_{k+1} - z_2\|^2 \\
&\quad + 2(t_{k+1} - 1) \langle x_{k+1} - x_k, z_2 - z_1 \rangle.
\end{aligned}
\end{equation}
Furthermore, since  
$$
\begin{aligned}
h_{k+1}&=\|x_{k+1} - z_1\|^2 - \|x_{k+1} - z_2\|^2 = \|z_2 - z_1\|^2 + 2 \langle x_{k+1} - z_2, z_2 - z_1 \rangle, \\
h_k&=\|x_k - z_1\|^2 - \|x_k - z_2\|^2 = \|z_2 - z_1\|^2 + 2 \langle x_k - z_2, z_2 - z_1 \rangle,
\end{aligned}
$$  
we have  
$$
h_{k+1} - h_k = 2 \langle x_{k+1} - x_k, z_2 - z_1 \rangle.
$$  
Combining with equation \cref{eq:point-convergence-result-2}, we obtain  
\begin{equation}\label{eq:point-convergence-result-3}
2s \left( \mathcal{E}_{k+1}(z_1) - \mathcal{E}_{k+1}(z_2) \right) = h_{k+1} + (t_{k+1} - 1)(h_{k+1} - h_k).
\end{equation}
By \cref{lem:Lyapunov-limit-exists}, the left-hand side of \cref{eq:point-convergence-result-3} has a limit as $ k \to \infty $, denoted as $ \ell $. Thus,  
$$
\lim_{k \to \infty} \left[ h_{k+1} + (t_{k+1} - 1)(h_{k+1} - h_k) \right] = \ell.
$$  
By \cref{lem:desecert} and the fact that $ \sum_{k=1}^\infty \frac{1}{t_{k+1} - 1} = +\infty $, it follows that $ \lim_{k \to \infty} h_k = \ell $.  

Now, consider subsequences $ \{\bar{x}_k\} $ and $ \{\widetilde{x}_k\} $ of $ \{x_k\} $ such that $ \lim_{k \to \infty} \bar{x}_k = z_1 $ and $ \lim_{k \to \infty} \widetilde{x}_k = z_2 $. Then,  
$$
\begin{aligned}
\ell &= \lim_{k \to \infty} \left( \|\bar{x}_k - z_1\|^2 - \|\bar{x}_k - z_2\|^2 \right) = -\|z_1 - z_2\|^2, \\
\ell &= \lim_{k \to \infty} \left( \|\widetilde{x}_k - z_1\|^2 - \|\widetilde{x}_k - z_2\|^2 \right) = \|z_1 - z_2\|^2.
\end{aligned}
$$  
Hence, $ \|z_1 - z_2\|^2 = 0 $, which implies $ z_1 = z_2 $. Therefore, there exists a unique cluster point $ x^* $ of $\{x_k\}$ such that $ \lim_{k \to \infty} x_k = x^* $.  

By \cref{thm:convergence-rate-and-sequence} and the lower semicontinuity of $ u_0(\cdot) $, we have  
$$
0 \le u_0(x^*) \le \liminf_{k \to \infty} u_0(x_k) = 0,
$$  
so $ u_0(x^*) = 0 $. By \cref{thm:weakpareto}, $ x^* $ is a weak Pareto solution of \cref{eq:MOP}.
\end{proof}

\section*{Acknowledgements} 
Thanks to Chengzhi Huang for his careful reading and insightful comments on the first draft of this article.
\section*{Funding}

\end{document}